\documentclass[12pt]{amsart}
\usepackage{hyperref}
\usepackage{amsfonts}
\usepackage{amsmath}
\usepackage{amssymb}
\usepackage{amscd}
\usepackage{graphicx}
\usepackage{latexsym}
\usepackage{color}
\usepackage{epsfig}
\usepackage{amsopn}
\usepackage{mathrsfs}
\usepackage{array}
\usepackage{soul}

\usepackage{booktabs}
\usepackage{longtable}
\theoremstyle{plain}
\newtheorem{lemma}{Lemma}[section]
\newtheorem*{theorem*}{Theorem}
\newtheorem*{lemma*}{Lemma}
\newtheorem*{proposition*}{Proposition}
\newtheorem*{conjecture*}{Conjecture}
\newtheorem*{corollary*}{Corollary}
\newtheorem*{problem*}{Problem}
\newtheorem{theorem}[lemma]{Theorem}
\newtheorem{conjecture}[lemma]{Conjecture}

\newtheorem{proposition}[lemma]{Proposition}

\newtheorem{question}[lemma]{Question}

\theoremstyle{definition}
\newtheorem{definition}[lemma]{Definition}
\newtheorem{example}[lemma]{Example}
\newtheorem{remark}[lemma]{Remark}

\newcommand{\CC}{\mathbb{C}}
\newcommand{\QQ}{\mathbb{Q}}
\newcommand{\RR}{\mathbb{R}}

\newcommand{\OO}{\mathcal{O}}

\newcommand{\cO}{\mathcal{O}}

\renewcommand{\P}{\mathbb{P}}
\newcommand{\cL}{\mathcal{L}}

\DeclareMathOperator{\Bl}{Bl}

\DeclareMathOperator{\edim}{edim}

\DeclareMathOperator{\mult}{mult}

\newcommand{\leqor}{\underset{{\scriptscriptstyle (}-{\scriptscriptstyle )}}{<}}

\begin{document}

\date{\today}
\author[{\L}. Farnik]{{\L}ucja Farnik}
\address{Department of Mathematics, Pedagogical University of Cracow,
   Podchor\c a\.zych 2,
   PL-30-084 Krak\'ow, Poland}
\email{Lucja.Farnik@gmail.com}

\author[K. Hanumanthu]{Krishna Hanumanthu}
\address{Chennai Mathematical Institute, H1 SIPCOT IT Park, Siruseri, Kelambakkam 603103, India}
\email{krishna@cmi.ac.in}

\author[J. Huizenga]{Jack Huizenga}
\address{Department of Mathematics, The Pennsylvania State University, University Park, PA 16802}
\email{huizenga@psu.edu}

\author[D. Schmitz]{David Schmitz}
\address{Mathematisches Institut,
   Universit\" at Bayreuth,
  D-95440 Bayreuth}
\email{schmitzd@mathematik.uni-marburg.de}

\author[T. Szemberg]{Tomasz Szemberg}
\address{Department of Mathematics, Pedagogical University of Cracow,
   Podchor\c a\.zych 2,
   PL-30-084 Krak\'ow, Poland}
\email{tomasz.szemberg@gmail.com}

\thanks{{\L}F was partially supported by the National Science Centre,
        Poland, grant 2018/28/C/ST1/00339.
        KH was partially supported by a grant from Infosys
        Foundation and by DST SERB MATRICS grant MTR/2017/000243.
        JH was partially supported by the  NSA\ Young Investigator Grant H98230-16-1-0306 and NSF FRG grant DMS 1664303.
        DS was partially supported by DFG grant PE 305/13-1.
        TS was partially supported by the National Science Centre Poland, grant 2018/30/M/ST1/00148}

\title{Rationality of Seshadri constants on general blow ups of $\P^2$}

\begin{abstract}
Let $X$ be a projective surface and let $L$ be an ample line bundle on
$X$. The global Seshadri constant $\varepsilon(L)$ of $L$ is defined
as the infimum of Seshadri constants $\varepsilon(L,x)$ as $x\in X$ varies.
It is an interesting question to ask if $\varepsilon(L)$ is a rational
number for any pair $(X, L)$. We study this question when
$X$ is a blow up of $\P^2$ at $r \ge 0$ very general points and $L$ is an
ample line bundle on $X$.
For each $r$ we define a \textit{submaximality threshold} which
governs the rationality or irrationality of $\varepsilon(L)$.
We state a conjecture which strengthens the SHGH Conjecture and
assuming that this conjecture is true we determine the submaximality threshold.

\end{abstract}

\maketitle

\section{Introduction}

Let $X$ be a smooth complex projective variety and let $L$ be a nef line bundle on
$X$.
The {\it Seshadri constant} of $L$ at $x \in X$ is defined as the real number
$$\varepsilon(X,L,x):=  \inf\limits_{\substack{x \in C}} \frac{L\cdot
  C}{{\rm mult}_{x}C},$$ where the infimum is taken over all
irreducible and reduced curves passing through $x$.
The Seshadri constants were defined by Demailly in \cite{Dem},
motivated by  the Seshadri
criterion for ampleness (\cite[Theorem 7.1]{Har70})
which says that $L$ is ample if and only if $\varepsilon(X,L,x) > 0$
for all $x \in X$.

Seshadri constants have turned out to be fundamental to the study of
positivity questions in algebraic geometry and a lot of research is
currently focused on problems related to Seshadri constants. One such
open problem is whether Seshadri constants can be irrational.

Assume that $X$ is a surface. If $L$ is an ample line bundle on $X$,
then for any $x \in X$, we have $0 < \varepsilon(X, L,x)
\le \sqrt{L^2}$. The first inequality is the Seshadri criterion for
ampleness and the second inequality is an easy observation. The largest and the smallest
values of Seshadri constants as the point $x$ varies are interesting
and generally they behave very differently.

To be more precise, one has the following two definitions:
\begin{align*}\varepsilon(X,L,1) &: = \sup\limits_{x\in X}
\varepsilon(X,L,x),\\
\varepsilon(X,L) &:= \inf\limits_{x \in X}  \varepsilon(X,L,x).
\end{align*}

It is known that $\varepsilon(X,L,1)=\varepsilon(X,L,x)$ for very general points $x\in X$ (see
\cite{Ogu}).  It is also expected
that  $\varepsilon(X,L,1) = \sqrt{L^2}$ in many situations.
For example, let $X$ be the blow up of $\P^2$ at at least 9 general
points. If  some well-known conjectures
are true, then there exist ample line bundles on $X$ such that
$\varepsilon(X,L,1) = \sqrt{L^2} \notin \QQ$.
 See \cite{DKMS,HH} for more details.

On the other
hand, $\varepsilon(X,L)$, called the global Seshadri constant, is usually attained at special points. In
this context, \cite[Question 1.6]{SS} asks whether $\varepsilon(X,L)$ is always
rational for any pair $(X,L)$. In this paper we study this question in
the case of blow ups of $\P^2$ at very general points.  On the one
hand, it is easy to exhibit ample line bundles $L$ such that
$\varepsilon(X,L)$ is rational.  On the other hand, we state a
strengthened version of the SHGH conjecture that implies that
$\varepsilon(X,L)$ can be irrational for some line bundles $L$ close to the boundary of the ample cone.  See Example
\ref{irrational-epsilon} for one such instance.

In fact, for $\mu\in \QQ$ we study uniform line bundles $L = L(\mu) = \mu H - \sum_i E_i$ on blow ups of $\P^2$ at
very general points and exhibit a threshold $\mu_0$ such that $\varepsilon(X, L)\in \QQ$ if $\mu \geq \mu_0$.  This is proved in Theorem \ref{thm-Rational}. We then state
Conjecture \ref{MainConjecture} which strengthens the SHGH
Conjecture.
Assuming this conjecture is true, we  show in Theorem
\ref{thm-Irrational}  that if $\mu < \mu_0$ then $\varepsilon(X, L) \notin \QQ$ unless $\sqrt{L^2}
\in \QQ$.

We will write $\varepsilon(L) = \varepsilon(X,L)$ when the variety $X$ is clear.

\subsection*{Acknowledgements} We thank
the Mathematisches Forschungsinstitut Oberwolfach for hosting
the Mini-Workshop  {\it Asymptotic
  Invariants of Homogeneous Ideals} during September 30 --
October 6, 2018, where most of this work was done.
The research stay of the se\-cond author was partially supported by the Simons Foundation
and by the Mathematisches Forschungsinstitut Oberwolfach and he is
grateful to them.  We would also like to thank the referee, whose comments helped improve the exposition of the paper.

\section{Sub-maximality threshold}\label{section1}

Let $p_1,\ldots,p_r\in \P^2$ be very general points and let $X =
\Bl_{p_1,\ldots,p_r} \P^2$ be the blowup of $\P^2$ at
$p_1,\ldots,p_r$.   Let $E_i$ be the exceptional divisor over $p_i$, and let $E = \sum_i E_i$.
Let $H$ denote the pull-back of $\cO_{\P^2}(1)$.

We will focus on {\it uniform} line bundles $L =
dH-mE$ on $X$, i.e., such where all exceptional divisors appear with the same multiplicity $m$. We are interested in the rationality or irrationality
of $\varepsilon(L)$. This only depends on the ratio $\mu=d/m$ and we work
with the $\QQ$-divisor $(d/m)H-E$. More generally, for $\mu\in \RR$, let $L(\mu)$ be the
$\RR$-divisor $\mu H -
E$.  If $L(\mu)$ is ample then $\mu > \sqrt{r}$. If $r \geq 10$, then the converse is true if the Nagata conjecture holds.

In this paper, we discuss the following question.

\begin{question}\label{question}
Let $\mu\in\QQ$ and suppose $L(\mu)$ is ample.  Is $\varepsilon (L(\mu))$ rational?
\end{question}

It is well-known that if $L$ is an ample $\QQ$-divisor and $\varepsilon(L,x) < \sqrt{L^2}$ then
$\varepsilon(L,x)$ is achieved by a curve $C$
containing $x$, and consequently, $\varepsilon(L,x) \in \QQ$.
So if $\varepsilon(L)$ is rational, then one of the following must be
true:

\begin{enumerate}\item $\varepsilon(L) = \sqrt{L^2} \in \QQ$, or

\item  $\varepsilon(L) <  \sqrt{L^2}$ and there is a pair $(C,x)$ where
$C$ is an irreducible and reduced curve containing a point $x$ such
that $$\varepsilon(L) =  \frac{L\cdot  C}{{\rm mult}_{x}C}.$$
\end{enumerate}
A curve $C$ satisfying $$\frac{L\cdot  C}{{\rm mult}_{x}C}\leqor \sqrt{L^2}$$  is called a {\it
  (weakly) submaximal curve} for $L$ with respect to $x$ (note that if
equality holds, then $\sqrt{L^2}$ is rational).
%If it satisfies the condition (2) above, it is called a {\it Seshadri curve}.
In light of this discussion, if $L$ is ample then we have $\varepsilon(L) \in \QQ$ if and only if either $\sqrt{L^2}\in \QQ$ or there is a weakly submaximal curve.

When the number $r$ of points is at most $9$, a complete answer to
Question \ref{question} is given in the following theorem.

\begin{theorem}\label{thm-delPezzo}
Let $r\leq 9$ and let $\mu\in \QQ$ be such that $L(\mu)$ is ample.  Then $\varepsilon(L(\mu)) \in \QQ$.
\end{theorem}
\begin{proof}
When $r\leq 8$, it is well-known that Seshadri constants of ample line bundles are rational at all points. See e.g. \cite[Remark 4.2]{Sa14}. More directly, it is also easy to exhibit weakly submaximal curves for $r \le 8$. See \cite[Example 2.4]{SS} for more details. For example, let $r=8$. In this case, $L(\mu)$ is ample if and only if $\mu > 17/6$. If $\mu \ge 3$, then an exceptional divisor $E_i$ and a point $x \in E_i$ give a weakly submaximal curve. Indeed, we have $$1 = L(\mu) \cdot E_i \leq \sqrt{L(\mu)^2} = \sqrt{\mu^2-8},$$ whenever $\mu \ge 3$.
For $\mu \in (17/6,3)$, let
$C$ be the sextic $6H-3E_1-2(E_2+ \ldots + E_8)$. This is a weakly submaximal curve for $L(\mu)$ if $6\mu - 17 \le \sqrt{\mu^2-8}$. This holds for $2.828\le \mu \le 3$. It follows that $\varepsilon(L(\mu)) \in \QQ$. Similarly, one can find  submaximal curves for ample bundles $L(\mu)$ when $r \le 8$.

For $r = 9$, the line bundle $L(\mu)$ is ample if and only if $\mu > 3$.  We show that there is a weakly submaximal curve for $L(\mu)$. First, if $\mu \geq \sqrt{10}$ then, as above, an exceptional divisor $E_i$ is a weakly submaximal curve for $L(\mu)$.

%Indeed, we have $$1 = L(\mu) \cdot E_i \leq \sqrt{L(\mu)^2} = \sqrt{\mu^2-9}$$ whenever $\mu \geq \sqrt{10}$.
If instead $\mu \in (3,\sqrt{10})$, we need to give a different weakly submaximal curve.  Consider the cubic $C = 3H - E$ through the 9 points, and let $x\in C$.  Then $C$ gives a weakly submaximal curve for $L(\mu)$ so long as $$3\mu-9=L(\mu)\cdot C \leq \sqrt{L(\mu)^2} = \sqrt{\mu^2 - 9},$$ and this inequality holds for $\mu \in (3,3.75]$.  Therefore $\varepsilon(L(\mu))\in \QQ$.
\end{proof}

Thus for the rest of the article we focus on the case $r \geq 10$.  We can shift our focus to the existence of weakly submaximal curves.

\begin{question}\label{question2}
For which real $\mu \geq \sqrt{r}$ does $L(\mu)$ admit a weakly submaximal curve?
\end{question}

The answer to Question \ref{question2} is perhaps most interesting when $\mu$ is rational and $L(\mu)$ is ample, but there is no difficulty in stating or studying it more generally as we have done above.  We first prove that there is a critical value $\mu_0 \geq \sqrt r$ such that
 $L$ admits a weakly submaximal curve if  $\mu \geq
\mu_0$. It follows that if $L(\mu)$ is ample and $\mu\in \QQ$ then $\varepsilon(L(\mu)) \in \QQ$ for $\mu \geq \mu_0$.

\begin{definition}
Let $X = \Bl_{p_1,\ldots,p_r} \P^2$ be the blowup of $\P^2$ at $r$ general points.  A real number $\mu_0 \geq \sqrt{r}$ is called the
\textit{submaximality threshold} for $r$ if
\begin{enumerate}
\item $L(\mu)$ does not admit a weakly submaximal curve for $\mu < \mu_0$, and \item $L(\mu)$ does admit a weakly submaximal curve for $\mu \geq \mu_0$.
\end{enumerate}
\end{definition}

In Section \ref{section2}, we prove that submaximality thresholds exist for
$r \ge 10$, assuming a strengthening of the SHGH Conjecture.
This in particular means that if
$\sqrt{r} < \mu < \mu_0$  and $\sqrt{L(\mu)^2} \notin \QQ$, then
$\varepsilon(L(\mu)) \notin \QQ$.
See Conjecture
\ref{MainConjecture} and Theorem \ref{thm-Irrational}.

%For $r\leq 8$ points, Question \ref{question} can be answered
%completely. In this case, it is well-known that  the Seshadri constants of any ample line
%bundle are rational. So $\varepsilon (L)$ rational for every $\mu > \sqrt{r}$.

%So, assume $r \geq 10$.  We will first show that there is a critical
%value $\mu_0 \geq \sqrt r$ such that $\varepsilon(L)$ is rational for
%$\mu_0 \le \mu$. More precisely, we have the following theorem.
%(what about for $\mu=\mu_0$?). We call $\mu_0$ the \emph{rationality threshold}.

\begin{theorem}\label{thm-Rational}
Let $r \ge 1$ and let $\mu\in \RR$. Then we have the following.
\begin{enumerate}
\item For any $r$, $L(\mu)$ admits a weakly submaximal curve for all $\mu \geq \sqrt{r+1}$.
\item If $r = 10$, then $L(\mu)$ admits a weakly submaximal curve for
  all $\mu \ge 77/24 \approx 3.208$.
\item If $r = 11$, then $L(\mu)$ admits a weakly submaximal curve for all $\mu \geq 4- \frac{\sqrt{3}}{3} \approx 3.422$.
\item If $r = 13$, then $L(\mu)$ admits a weakly submaximal curve for all
 $\mu \geq \frac{1}{6}(26-\sqrt{13}) \approx 3.732$.
\end{enumerate}
\end{theorem}

Provided the submaximality threshold $\mu_0$ for $r$ exists, Theorem \ref{thm-Rational} can be viewed as giving a lower bound for $\mu_0$.

\begin{proof} (1) As in the proof of Theorem \ref{thm-delPezzo}, an exceptional divisor $E_i$ and a point $x\in E_i$ give the required weakly submaximal curve.

(2) For $r=10$, consider the complete linear system
$$\cL =  |10H - 4E_1 - 3\sum_{i=2}^{10} E_i|.$$
It is known that this system is non-special, since \cite{CM1} proves the SHGH Conjecture for all \textit{quasi-homogeneous} systems of the form $|dH-nE_1-m\sum_{i=2}^r E_i|$, when $m \le 3$. In particular, the linear system $\cL$
is a pencil, since its expected dimension is 1. An equivalent version of the SHGH Conjecture \cite[Conjecture 3.4]{G1} says that only possible fixed curves of a non-special pencil are $(-1)$-curves.
See the next section for a discussion about the various formulations of the SHGH Conjecture.

%The expected dimension of this linear system is $1$, so we can pick a pencil of curves of this class (in fact the system is nonspecial and there is precisely a pencil of curves of this class, but we will not need this).

We claim the pencil $\mathcal{L}$ has a singular member.  Suppose, on the contrary, that all members of the pencil are smooth. We first claim that $\mathcal{L}$ has no fixed curves. This is clear if the generic member of $\mathcal{L}$ is irreducible. Otherwise, every member of the pencil is disconnected, since $\mathcal{L}$ consists only of smooth curves. If $C$ is a fixed curve, then by the observation in the previous paragraph, $C$ is a $(-1)$-curve. Since members of $\mathcal{L}$ are smooth, we have
$\mathcal{L} \cdot C = C^2 < 0$. But this is not possible, since $\mathcal{L}$ is in standard form and hence has non-negative intersection with all $(-1)$-curves (see \cite{HH}).

Resolve the indeterminacy locus of $\phi_\cL:\P^2\dashrightarrow \P^1$ by blowing up $k$ (possibly infinitely near) points to obtain a morphism $Y\to \P^1$.  Note that all members of $\mathcal{L}$ are smooth curves of genus $3 = \binom{9}{2}-\binom{4}{2}-9\binom{3}{2}$. Hence their pull-backs to $Y$ are also smooth of genus 3 and topological Euler characteristic $-4$.   Then $\chi_{\text{top}}(Y) = 2\cdot (-4) = -8$ (see \cite[Theorem 7.17]{EH}), but also $$\chi_{\text{top}}(Y) = \chi_{\text{top}}(\P^2) + k = 3+ k.$$ This contradiction shows that there must be a singular member of the pencil.

Let $C$ be a singular member of this pencil with singularity at $x\in C$. Then $(C,x)$ gives a weakly submaximal curve for $L(\mu)$ if $$\frac{10\mu-31}{2}=\frac{L(\mu)\cdot C}{2} \leq \sqrt{L(\mu)^2} = \sqrt{\mu^2-10},$$ and this inequality holds for $\mu\in \left[\frac{77}{24},\frac{13}{4}\right].$

Since $\frac{13}{4} < \sqrt{11}$, we need to give a different weakly
submaximal curve for $L(\mu)$ when $\mu\in (\frac{13}{4},\sqrt{11})$.
Consider a cubic through 9 of the 10 points, as in the proof of
Theorem \ref{thm-delPezzo} in the $r=9$ case.  This gives a weakly
submaximal curve for $L(\mu)$ if $$3\mu - 9=L(\mu) \cdot C \leq \sqrt{L(\mu)^2} = \sqrt{\mu^2-10},$$ and this inequality holds for $\mu\in \left[\frac{13}{4},\frac{7}{2}\right]$.  Thus, $L(\mu)$ admits a weakly submaximal curve for all $\mu \geq \frac{77}{24}$.

(3) For $r=11$, there is a pencil of curves of class $$4H - 2E_1 -
\sum_{i=2}^{11} E_i.$$ By a similar computation as in the case $r=10$, this pencil contains a singular curve $C$ with a singular point $x\in C$.
The pair $(C,x)$ gives a weakly submaximal curve if $$\frac{4\mu-12}{2} = L(\mu)\cdot C \leq \sqrt{L(\mu)^2} = \sqrt{\mu^2-11},$$ and this inequality holds for $\mu\in [4-\frac{\sqrt{3}}{3},4+\frac{\sqrt{3}}{3}].$ Since $4 + \frac{\sqrt{3}}{3} > \sqrt{12}$, we are done.

(4)  Finally, for $r=13$, there is a pencil of curves of class $$4H -
\sum_{i=1}^{13} E_i.$$  Again as above, the pencil has a singular member $C$ with singularity $x\in C$.  It gives a weakly submaximal curve so long as $\mu\in \left[\frac{1}{6}(26-\sqrt{13}),\frac{1}{6}(26+\sqrt{13})\right]$, and since $\frac{1}{6}(26+\sqrt{13}) > \sqrt{14}$ we are done.
\end{proof}

\iffalse
\begin{remark}
These inequalities correspond to the following interesting curve classes.
\begin{enumerate}
\item For $r=10$, there is a pencil of curves of class $$10H - 4E_1 - 3\sum_{i=2}^{10} E_i.$$ In this pencil there is a curve which has a singularity of multiplicity $2$ at some point. (Check)

\item For $r=11$, there is a pencil of curves of class $$4H - 2E_1 - \sum_{i=2}^{11} E_i.$$ In this pencil there is a curve which has a singularity of multiplicity $2$ at some point. (Check)

\item For $r=13$, there is a pencil of curves of class $$4H - \sum_{i=1}^{13} E_i.$$  In this pencil there is a curve which has a singularity of multiplicity $2$ at some point. (Check)

\item For all other $r$, the strongest inequalities are given by an exceptional divisor $E_1$. (Explain)
\end{enumerate}
\end{remark}
\fi

\section{A generalized SHGH conjecture}

In Theorem \ref{thm-Rational}, we established upper bounds on the
submaximality threshold.
Conversely, to produce lower bounds on the submaximality threshold it is
necessary to show that there are no weakly submaximal curves.  We state a generalization of the SHGH conjecture which would guarantee that such curves cannot exist.

\subsection{The SHGH conjecture} Suppose that we have integers $d \geq 0$ and $m_1,\ldots,m_r\geq 0$. Consider the linear series $$\cL = |dH - m_1E_1 -\cdots -m_rE_r|$$ on a general blowup $X = \Bl_{p_1,\ldots,p_r}\P^2$. The \emph{expected dimension} of the series is defined to be
$$\edim \cL=\max \left\{{d+2\choose 2} - \sum_i {m_i+1\choose 2}-1,-1\right\},$$ and the series is \emph{nonspecial} if $\dim \cL = \edim \cL$.  There are many statements equivalent to the SHGH conjecture, but the following version is relevant for our purposes.

\begin{conjecture}[SHGH]\label{conj-SHGH}
If $\cL$ is special, then every divisor in $\cL$ is nonreduced.
\end{conjecture}

The contrapositive statement ``if there is a reduced curve in $\cL$
then $\cL$ is nonspecial'' is also often useful.
Also note that if we add a very general simple point to the linear
system $\cL$, then the dimension and expected dimension drop exactly by 1. More
precisely, we have $\dim \cL' = (\dim \cL) - 1$ and
$\edim \cL' = (\edim \cL) - 1$, where $\cL'$ is the
linear system $|dH - m_1E_1 -\cdots -m_rE_r-E_{r+1}|$ on
a general blow up $\Bl_{p_1,\ldots,p_r,p_{r+1}}\P^2$.
Hence if Conjecture \ref{conj-SHGH} is only stated for systems with $\edim \cL = -1$, then by imposing additional simple points the full conjecture follows.

More refined versions of Conjecture \ref{conj-SHGH} discuss the structure of the base locus of $\cL$ more carefully and seek to completely classify the special systems.  These various refinements have been stated and compared by various authors including Segre \cite{S}, Harbourne \cite{H1}, Gimigliano
\cite{G} and Hirschowitz \cite{Hi}. The various formulations are
equivalent. See \cite{CM2,H2} for more details.

The following stronger version of the SHGH conjecture easily follows from a conjecture attributed to Hirschowitz in
\cite[Conjecture 4.9]{C}. It is also mentioned in \cite[Conjecture 3.1 (iv)]{CM2}.

\begin{conjecture}\label{conj-Hirschowitz}
If the general curve $C\in \cL$ is reduced, then $\cL$ is nonspecial and $C$ is smooth on $X$.
\end{conjecture}

More precisely, a slightly weaker version of the original conjecture from \cite{C} reads as follows.

\begin{conjecture}[Hirschowitz {\cite[Conjecture 4.9]{C}}]\label{conj-Hirschowitz2}
Suppose $\cL$ is nonempty and nonspecial, and let $C\in \cL$ be general.  Suppose $p_a(C) \geq 0$ and $C$ is reduced.  Then $C$ is smooth and irreducible on $X$.
\end{conjecture}

\begin{remark}
Let us show that Conjectures \ref{conj-SHGH} and \ref{conj-Hirschowitz2} imply Conjecture \ref{conj-Hirschowitz}.  By imposing additional simple points, it suffices to check Conjecture \ref{conj-Hirschowitz} in the case where $\edim \cL = 0$.  Let $C\in \cL$ be general and suppose it is reduced.  By Conjecture \ref{conj-SHGH}, $\cL$ is nonspecial.  If $C$ is irreducible, then $p_a(C)\geq 0$ and $C$ is smooth by Conjecture \ref{conj-Hirschowitz2}.  Suppose $C$ is not irreducible.  Then $C = C' + C''$ for some curves $C'\in \cL'$ and $C''\in \cL''$.  Since $\edim \cL = 0$ and $C$ is reduced, we have $\cL = \{C\}$ and therefore $\cL' = \{C'\}$ and $\cL'' = \{C''\}$.  By Conjecture \ref{conj-SHGH}, we have $\edim \cL' = \edim \cL'' = 0$ and $$\edim \cL = \edim \cL' + \edim \cL'' + C'\cdot C''.$$ Therefore $C'\cdot C'' = 0$, and if $C'$ and $C''$ are smooth then so is $C$.  By induction on the number of irreducible components, $C$ is smooth.
\end{remark}

\subsection{A generalized SHGH conjecture} We now state a stronger SHGH conjecture by studying the loci in $\cL = |dH - m_1E_1 - \cdots -m_rE_r|$ of curves with a singularity of some multiplicity $t\geq 2$.  Fix a point $x \in X$.  Then the expected codimension in $\cL$ of curves with a singularity of multiplicity $t$ at $x$ is ${t+1\choose 2}$.  As the point $x\in X$ varies, the expected codimension in $\cL$ of curves with a singularity of multiplicity $t$ at some point is ${t+1\choose 2}-2$.

Various examples show that it is too much to hope for that the locus in $\cL$ of curves with a $t$-uple point always has the expected codimension.  But, the source of these counterexamples seems to be nonreduced curves in the series.

\begin{example}
For example, let $r = 8$ and consider the series $$\cL = |6H - 2\sum_{i=1}^8 E_i|.$$  The SHGH conjecture implies that $\dim \cL = 28 - 24  -1 = 3$.  The expected codimension in $\cL$ of curves with a $4$-uple point is ${5\choose 2} - 2 = 8$, so we would expect that there are not any such curves.  On the other hand, in the pencil of cubics through the 8 points there is a singular cubic, and its square is a member of $\cL$ with a $4$-uple point.
\end{example}

In general, the locus in $\cL$ of nonreduced curves can be quite large and contain highly singular curves, but it seems possible that this is the only source of unexpectedly singular curves in linear series.  We  make the following conjecture.

\begin{conjecture}\label{MainConjecture}
Let $X$ be a blow up of $\P^2$ at $r\ge 0$ very general
points. Suppose $d \geq 1$, $t\geq 1$, and $m_1,\ldots,m_r\geq 0$ are
integers such that $${d+2 \choose 2} -\sum_{i=1}^r {m_i+1 \choose 2} \le \max\left\{{t+1\choose 2} - 2,0\right\}.$$
Then any curve $C\in |dH - m_1E_1 -\cdots -m_rE_r|$ which has a point of multiplicity $t$ is non-reduced.
\end{conjecture}

Some initial cases of Conjecture \ref{MainConjecture} are well-known.  In particular, the case $t=1$ is equivalent to the $\edim \cL = -1$ case of Conjecture \ref{conj-SHGH}, so it is equivalent to Conjecture \ref{conj-SHGH}.  When $t=2$, the conjecture is the $\edim \cL = 0$ case of Conjecture \ref{conj-Hirschowitz}, so it is equivalent to Conjecture \ref{conj-Hirschowitz}.

\begin{remark}
We could weaken Conjecture \ref{MainConjecture} by changing the conclusion to ``Then any curve $C\in \cL$ which has a point of multiplicity $t$ is non-reduced or non-irreducible.''  This weakened version would still be strong enough to carry out the arguments in the next section.  We highlight the stronger version instead since it is more analogous to the SHGH and Hirschowitz conjectures \ref{conj-SHGH} and \ref{conj-Hirschowitz}.
\end{remark}

\section{The submaximality threshold for 10 or more points}\label{section2}
%\section{The conjecture computes the submaximality threshold}\label{section2}

For the rest of the paper, we assume that Conjecture \ref{MainConjecture} is true.  Under this assumption, we prove that Theorem \ref{thm-Rational} is sharp.

\begin{theorem}\label{thm-Irrational}
Suppose Conjecture \ref{MainConjecture} is true, and let $r\geq 10$.  Then the submaximality threshold $\mu_0$ for $r$ exists, and $$\mu_0 = \begin{cases} \frac{77}{24} &\textrm{if }r=10\\ 4- \frac{\sqrt{3}}{3} &\textrm{if } r=11\\ \frac{1}{6}(26-\sqrt{13}) & \textrm{if } r=13\\ \sqrt{r+1} & \textrm{if } r= 12 \textrm{ or } r\geq 14.\end{cases}$$
\end{theorem}
\begin{proof}
Let $\mu_0$ be the number in the statement, and let $\mu$ be a number with $\sqrt{r} < \mu < \mu_0$.  By Theorem \ref{thm-Rational} we need to show there is no weakly submaximal curve for $L(\mu)$.  If there is a weakly submaximal curve for $L(\mu)$ then there is an irreducible and reduced curve $C$ and a point
$x\in C$ such that $$\frac{L(\mu) \cdot C}{\mult_x C} \leq \sqrt{L(\mu)^2}.$$
Since $\mu< \sqrt{r+1}$, the curve $C$ is not an exceptional divisor $E_i$, so
$$\OO_X(C) = \OO_X(dH - \sum m_i E_i)$$ with $d > 0$ and $m_i\geq 0$.
Let $t = \mult_x C$, so $1\leq t \leq d$.  Then by Conjecture \ref{MainConjecture}
we have the simultaneous
inequalities \begin{align}\label{eqn-rationality}\tag{$\ast$} \frac{\mu
               d-\sum_i m_i}{t} & \le  
                                      \sqrt{\mu^2-r}\\\label{eqn-edim}\tag{$\ast\ast$}{d+2\choose
               2} - \sum_{i=1}^r {m_i+1\choose 2} &> \max\left\{
                                                      {t+1\choose 2}-2
                                                      ,
                                                      0\right\}.\end{align}
                                                      
                                                      We furthermore claim that we may assume $t<d$.  Since $C$ is reduced and irreducible, if $t=d$ then $t=d=1$.  In that case (\ref{eqn-edim}) shows $\sum_i m_i \leq 2$, and (\ref{eqn-rationality}) gives $\mu \geq 1 + \frac{r}{4}.$  But this contradicts $\mu < \sqrt{r+1}$.  
                                                      
                                                       In Proposition \ref{prop-irrational} we will show that since $\mu < \mu_0$ these inequalities cannot be satisfied.
\end{proof}

The main work in the proof of Theorem \ref{thm-Irrational} then lies in Proposition \ref{prop-irrational}, which is essentially numerical.  To avoid repeating our assumptions we make the following definition.

\begin{definition}
A \emph{test pair} $(C,t)$ consists of a curve class $C = dH - \sum_{i=1}^r m_iE_i$, where $d\geq 2$ and $m_i\geq 0$ are integers, and an integer $t$ satisfying $1\leq t < d$.
\end{definition}

Notice that if $(C,t)$ is a test pair satisfying (\ref{eqn-edim}) then the curve class $C$ is effective, since the expected dimension of the linear series $|C|$ is nonnegative.

\begin{proposition}\label{prop-irrational} Let $r\geq 10$, and let $\mu_0$ be the number in the statement of Theorem \ref{thm-Irrational}.  Suppose $\mu$ is a number with $\sqrt{r} < \mu < \mu_0$.  There is no test pair $(C,t)$  satisfying (\ref{eqn-rationality}) and (\ref{eqn-edim}).
\end{proposition}

\subsection{Bounding the multiplicities}  Suppose $(C,t)= (dH - \sum_i m_i E_i,t)$ is a test pair satisfying (\ref{eqn-rationality}) and (\ref{eqn-edim}), and let $\overline m = \frac{1}{r}\sum_i m_i\in \QQ$ be the average multiplicity.  In this section we bound $\overline m$ and $t$ uniformly in terms of $r$, in order to decrease the search space for counterexamples to Proposition \ref{prop-irrational}.

From (\ref{eqn-rationality}) and (\ref{eqn-edim}) and Cauchy-Schwarz
we conclude
\begin{eqnarray}\label{eqn-rationality2}
\frac{\mu d - r\overline m}{t} &\le& \sqrt{\mu^2-r}\\ \label{eqn-edim2}
(d+2)(d+1)-r(\overline m+1)\overline m & >&  (t+1)t-4.
\end{eqnarray}
Rearrange (\ref{eqn-rationality2}) to get $$d \le \frac{r\overline m+t\sqrt{\mu^2-r}}{\mu}.$$ Now we substitute this inequality into (\ref{eqn-edim2}) and rearrange the terms to prove the following quadratic inequality in $\overline m$ and $t$.
\begin{lemma}
If $(C,t)$ is  a test pair satisfying (\ref{eqn-rationality}) and (\ref{eqn-edim}) with average multiplicity $\overline m$, then the quadratic expression 
\begin{align*}Q(\overline m,t)&:=\left(\frac{r^2}{\mu^2} - r\right) \overline m^2 + \frac{2r \sqrt{\mu^2-r}}{\mu^2} \overline mt- \frac{r}{\mu^2} t^2\\&\quad + \left( \frac{3r}{\mu}-r\right)\overline{m}+\left(\frac{3 \sqrt{\mu^2-r}}{\mu}-1\right) t + 6\end{align*} satisfies $Q(\overline m,t) > 0$.  Therefore, the point $(\overline m,t)$ lies in the region $\Omega$ of the $(\overline m,t)$-plane defined by the inequalities $t\geq 1$, $\overline m \geq 0$, and $Q(\overline m,t) > 0$.
\end{lemma}
 Now we analyze the region $\Omega$ more carefully.
\begin{lemma}\label{lem-multiplicityBound}
Let $r \geq 10$.  If $\sqrt{r} \leq \mu \leq \sqrt{r+1}$, then the region $\Omega$ in the $(\overline m,t)$-plane is bounded.  In particular, $\Omega$ is contained in the strip defined by the inequalities $$0 \leq \overline m \leq \frac{25}{4r-12\sqrt{r}},$$
and if $t$ is an integer then
$$\begin{array}{ll} t\in \{1,2,3,4,5\} & \textrm{if $r=10$}\\
t\in \{1,2,3,4\} & \textrm{if $r=11$}\\
t\in \{1,2,3\} & \textrm{if $r=12$}\\
t\in \{1,2\} & \textrm{if $r\geq 13$}.
\end{array}$$
\end{lemma}
\begin{proof}
The equation $Q(\overline m,t) = 0$ defines a parabola in the
$(\overline m,t)$-plane, since the discriminant of the homogeneous
degree $2$ part is
$$\left(\frac{2r\sqrt{\mu^2-r}}{\mu^2}\right)^2+4\left(\frac{r^2}{\mu^2}-r\right)\frac{r}{\mu^2}=0.$$
Observe that the point $(\overline m,t) = (0,1)$ is in $\Omega$, since $$Q(0,1) = 5-\frac{r}{\mu^2}+\frac{3\sqrt{\mu^2-r}}\mu >0$$ since $\mu > \sqrt{r}$.

Next we establish the bound on $\overline m$. View $\overline m > 0$ as fixed and consider the discriminant $\Delta_t(\overline m)$ of the polynomial $Q(\overline m,t)$ of $t$:
$$\Delta_t(\overline m) = \frac{1}{\mu^2}\left(-(4r^2-12r\mu + 4r \sqrt{\mu^2-r})\overline m+(15r+10\mu^2-6\mu\sqrt{\mu^2-r})\right)$$
Then $\Delta_t(\overline m)$ is decreasing in $\overline m$ since $r
\geq 10$ and $\mu^2 < r+1$, and $\Delta_t(0) > 0$.  For $$\overline
m_0(\mu) :=
\frac{15r+10\mu^2-6\mu\sqrt{\mu^2-r}}{4r^2-12r\mu+4r\sqrt{\mu^2-r}}>0,$$
we have $\Delta_t(\overline m_0(\mu)) = 0$, so the parabola
$Q(\overline m,t)=0$ is tangent to and left of the vertical line $\overline{m} =
\overline m_0(\mu).$  The numerator in the quotient defining
$\overline m_0(\mu)$ is decreasing in $\mu$ on
$[\sqrt{r},\sqrt{r+1}]$, and the denominator in the quotient is
increasing in $\mu$ on $[\sqrt{r},\sqrt{r+1}]$.
This can be seen by differentiating the numerator and denominator with
respect to $\mu$ and determining the signs of the derivatives on
$[\sqrt{r},\sqrt{r+1}]$.
Thus $\overline
m_0(\mu)$ is maximized on $[\sqrt{r},\sqrt{r+1}]$ when $\mu = \sqrt{r}$, and for $\mu \in [\sqrt{r},\sqrt{r+1}]$ we have $$\overline m_0(\mu) \leq \frac{25}{4r-12\sqrt{r}}.$$ Thus the region $\Omega$ lies left of the line $\overline m = 25/(4r-12\sqrt{r})$.

Suppose $t_0>1$ is a number such that $Q(\overline m,t_0) < 0$ for all $\overline m\geq 0$.  Since $Q(0,1) > 0$, the parabola $Q(\overline m,t)=0$ crosses the $t$-axis at a point $(0,t_1)$ between $(0,1)$ and $(0,t_0)$.  Since the parabola is tangent to $\overline m = \overline m_0(\mu)$ at some point, the only possibility is that the point of tangency lies below the line $t = t_0$.  Then $\Omega$ is contained in the half-space $t \leq t_0$.

Thus to complete the proof, we must show that for all $\overline m \geq 0$ and $\sqrt{r} \leq \mu \leq \sqrt{r+1}$,
$$\begin{array}{ll}
Q(\overline m,6) < 0 & \textrm{if $r=10$}\\
Q(\overline m,5) < 0 & \textrm{if $r=11$}\\
Q(\overline m,4) < 0 & \textrm{if $r=12$}\\
Q(\overline m,3) < 0 & \textrm{if $r\geq 13$}.
\end{array}
$$
Proving these inequalities is best left to the computer; for a given $r$ and $t_0$ it is straightforward to maximize $Q(\overline m,t_0)$ on the region of $(\overline m,\mu)$ with $\overline m\geq 0$ and $\sqrt{r} \leq \mu \leq \sqrt{r+1}$. We carried this out to check the inequalities for $r \leq 19$.

Once $r\geq 20$, we can give a straightforward argument.  For $\overline m\geq 0$ and $\sqrt r \leq \mu \leq \sqrt{r+1}$, we compute
\begin{align*}-Q(\overline m,3) &= \left(r-\frac{r^2}{\mu^2}\right)\overline m^2+\left(r-\frac{3r}{\mu}-\frac{6r \sqrt{\mu^2-r}}{\mu^2}\right)\overline m +\left(-3-\frac{9\sqrt{\mu^2-r}}{\mu}+\frac{9r}{\mu^2}\right)\\
&\geq  \left(r-\frac{r^2}{r^2}\right)\overline m^2+\left(r - \frac{3r}{\sqrt{r}}-\frac{6r}{r}\right)\overline m+\left(-3 - \frac{9}{\sqrt{r}}+\frac{9r}{r+1}\right)\\
&= (r-3\sqrt{r}-6)\overline m+ \left(\frac{9r}{r+1}-\frac{9}{\sqrt{r}}-3\right).
\end{align*} Both coefficients of this linear polynomial are positive since $r\geq 20$, so $Q(\overline m,3) < 0$ for all $\overline m \geq 0$.
\end{proof}

\subsection{Balanced pairs} Suppose the test pair $(C,t) = (dH - \sum_i m_i E_i,t)$ satisfies (\ref{eqn-rationality}) and (\ref{eqn-edim}).  Write the multiplicities in decreasing order $m_1\geq m_2\geq \cdots \geq m_r$.  If $m_1 - m_r \geq 2$, we can replace $m_1$ by $m_1-1$ and $m_r$ by $m_r+1$.  Then the resulting test pair still satisfies (\ref{eqn-rationality}) and (\ref{eqn-edim}).  Thus, if Proposition \ref{prop-irrational} is false, we can find a test pair $(C,t)$ satisfying (\ref{eqn-rationality}) and (\ref{eqn-edim}) where $C$ is a \emph{balanced curve class} of the form \begin{equation}\label{eqn-balanced} dH - m(E_1+\cdots +E_s) - (m-1)(E_{s+1}+\cdots + E_r)\end{equation} We can compactly record a balanced class by the tuple $(d;m^s,(m-1)^{r-s})$, where $s>0$ is as in (\ref{eqn-balanced}).  We call a test pair $(C,t)$ a \emph{balanced pair} if $C$ is balanced.

Given a balanced pair satisfying (\ref{eqn-edim}), we can easily check if it is a counterexample to Proposition \ref{prop-irrational}.
\begin{lemma}\label{lem-balancedCriterion}
Let $(C,t)= ((d;m^s,(m-1)^{r-s}),t)$ be a balanced pair satisfying (\ref{eqn-edim}), and let $$M = sm+(r-s)(m-1)=r\overline m$$ and $$\Delta = M^2-r(d^2-t^2).$$ Then the balanced pair is not a counterexample to Proposition \ref{prop-irrational} if  either 
\begin{itemize}
\item $\Delta <0$, or
\item $\Delta \geq 0$, and the number $$\mu_- = \frac{dM - t\sqrt{\Delta}}{d^2-t^2}$$ satisfies $\mu_- \geq \mu_0$.
\end{itemize}
\end{lemma}
\begin{proof}
Inequality (\ref{eqn-rationality}) reads $$d\mu-M \le t \sqrt{\mu^2 -r}.$$ Both sides of the inequality are positive since $C$ is effective, so squaring both sides and rearranging shows this is equivalent to \begin{equation}\label{eqn-equivalent} R(\mu):= (d^2-t^2)\mu^2-2dM \mu+(M^2+t^2r) \le 0.\end{equation}
Since $t< d$, the graph of $R(\mu)$ is an upward parabola.  The discriminant of the quadratic polynomial $R(\mu)$ is $4t^2\Delta$.  Therefore inequality (\ref{eqn-equivalent}) is false for $\mu < \mu_0$ if either $R(\mu) = 0$ has no real roots (and $\Delta < 0$), or if the smaller root (which is $\mu_-$) is at least $\mu_0$.
\end{proof}

\subsection{Critical pairs} We make one further reduction to further limit the search space for counterexamples to Proposition \ref{prop-irrational}.  Let $(C,t)$ be a balanced pair satisfying (\ref{eqn-rationality}) and (\ref{eqn-edim}).  If we can increase the smallest multiplicity $m_r$ by $1$ without making (\ref{eqn-edim}) false, then inequality (\ref{eqn-rationality}) still holds.  Similarly, if $t<d-1$ and we can increase $t$ by $1$ without making (\ref{eqn-edim}) false, then again inequality (\ref{eqn-rationality}) still holds.  We call a balanced pair $(C,t)$ a \emph{critical pair} if (\ref{eqn-edim}) is true but: \begin{itemize}
\item increasing $m_r$ by 1 makes (\ref{eqn-edim}) false, and 
\item either $t=d-1$, or increasing $t$ by 1 makes (\ref{eqn-edim}) false.
\end{itemize}
Thus, if Proposition \ref{prop-irrational} is false, then there is a counterexample $(C,t)$ which is a critical pair.
%\end{lemma}

\begin{proposition}
Proposition \ref{prop-irrational} is true for $10\leq r \leq 19$.
\end{proposition}
\begin{proof}
Fix some $r$ with $10\leq r \leq 19$.  Given integers $d\geq 1$ and $t \geq 1$, there is at most one critical pair $$((d;m^s,(m-1)^{r-s}),t).$$  Since Lemma \ref{lem-multiplicityBound} bounds $t$ and the average multiplicity $\overline m = \frac{1}{r} \sum m_i$ of any counterexample to Proposition \ref{prop-irrational}, there are only finitely many critical pairs which are potentially counterexamples.  We programmed a computer to list them all. For each  pair, Lemma \ref{lem-balancedCriterion} shows that the pair is not a counterexample to Proposition \ref{prop-irrational}. 
\end{proof}

We give more detail in the case $r=12$.

\begin{example} Let $r=12$.
In Table \ref{table-12}, we list all the critical pairs $((d;m^s,(m-1)^{12-s}),t)$ which are consistent with Lemma \ref{lem-multiplicityBound}.  According to the lemma, $t\in \{1,2,3\}$ and the total  multiplicity $M$ is bounded by $46$.  For each $t$, we increase $d$ and list any corresponding critical pair until $M$ would exceed this bound.  In the notation of Lemma \ref{lem-balancedCriterion} we then compute the number $\Delta$, and if $\Delta \geq 0$ we compute $\mu_-$.  By Lemma \ref{lem-balancedCriterion}, if $\Delta < 0$ or if $\Delta \geq 0$ and $\mu_- \geq \mu_0$ then the critical pair is not a counterexample.  In each case where $\Delta \geq 0$, we observe $\mu_- = 4 > \sqrt{13} = \mu_0$.
This proves Proposition \ref{prop-irrational} for $r=12$.
\begin{table}\caption{Critical pairs for $r=12$.}\label{table-12}
\begin{tabular}{ccccccccccc}
$C$ & $t$ & $M$ & $\Delta$ &$\mu_-$&\qquad\qquad&$C$ & $t$ & $M$ & $\Delta$ &$\mu_-$ \\\hline
$(2;1^5)$ &  1 & 5 & $-11$ &&&$(11;4^1,3^{11})$ &2 & 37 & $-35$ &\\
$(3;1^9)$ &  1 & 9 & $-15$ &&&$(12; 4^4,3^8)$ & 2 & 40 & $-80$ &\\
$(4;2^1,1^{11})$ & 1 & 13  & $-11$ &&&$(4; 1^{10})$ &3 & 10 & 16 &4\\
$(5;2^4,1^8)$ & 1  & 16 & $-32 $&&&$(5;2^2,1^{10})$ & 3 & 14 &4 &4 \\
$(9;3^6,2^6)$ & 1  & 30 & $-60$ &&&$(6;2^5,1^7)$ &3 & 17 & $-35$& \\
$(13; 4^8,3^4)$ & 1 & 44 & $-80$ &&&$(7;2^9,1^3)$ &3 & 21 & $-39$&  \\
$(3;1^8)$ & 2 & 8 & 4 &4&&$(8;3^1,2^{11})$ &3 & 25 & $-35$ &\\
$(4;1^{12})$ & 2 & 12 & 0 &4&&$(9;3^4,2^8)$ &3 & 28 & $-80$ &\\
$(5;2^3,1^9)$ & 2 & 15 &$-27$ &&&$(10;3^8,2^4)$ &3 & 32 & $-68$ &\\
$(6;2^7,1^5)$ & 2 & 19 & $-23$ &&&$(11;3^{12})$ &3 & 36 & $-48$&\\
$(7;2^{11},1^1)$ & 2 & 23 & $-11$&&&$(12;4^3,3^9)$ &3 & 39 & $-99$ &\\
$(8;3^2,2^{10})$ & 2 & 26 & $-44$&&&$(13;4^7,3^5)$ &3 & 43 & $-71$ &\\
$(9;3^5,2^7)$ & 2 & 29 & $-83$&&&$(14;4^{10},3^2)$ & 3 & 46 & $-128$ & \\
$(10;3^9,2^3)$ & 2 & 33 &$-63$ &
\end{tabular}
\end{table}
\end{example}
On the other hand, once $r\geq 20$ we can give an argument that requires minimal computation.

\begin{proposition}
Proposition \ref{prop-irrational} is true for $r\geq 20$.
\end{proposition}
\begin{proof}
Suppose a critical pair $(C,t) = ((d;m^s,(m-1)^{r-s}),t)$ violates
Proposition \ref{prop-irrational}.  Then Lemma
\ref{lem-multiplicityBound} shows $t\in \{1,2\}$ and $\overline m <
1$.  For the last inequality, we use the hypothesis $r \ge 20$.
Therefore $m=1$ and $M = s<r$.

Note that the inequality \eqref{eqn-edim} must be as sharp as possible for
$((d;1^s,0^{r-s}),t)$; in other words, we have an equality
$${d+2\choose 2} - M = \max\left\{
{t+1\choose 2}-2,0\right\}+1.$$ Indeed, if this fails then the
inequality \eqref{eqn-edim} is also satisfied by
$((d;1^{s+1},0^{r-s-1}),t)$, which contradicts the hypothesis that
$((d;1^s,0^{r-s}),t)$ is critical.

Since $t \in \{1,2\}$, it follows that $$ M = \frac{(d+2)(d+1)}{2} -t.$$  But then we claim that \begin{align*} \Delta: = M^2 - r(d^2-t^2) < 0,\end{align*} so that the pair is not a counterexample by Lemma \ref{lem-balancedCriterion}.  If $d <5$ then the only critical pairs are $((2;1^5),1)$, $((3;1^9),1)$, $((4;1^{14}),1)$, $((3;1^8),2)$, and $((4;1^{13}),2)$, and the inequality holds in these cases since $r\geq 20$.   So, assume $d\geq 5$.

Now since $t\in \{1,2\}$ and $d\geq 5$, $$\frac{d^2-t^2}{M} >\frac{2(d^2 - 4)}{(d+2)(d+1)} = \frac{2(d-2)}{(d+1)} \geq 1 > \frac{M}{r},$$ and therefore $M^2-r(d^2-t^2)<0$.
\end{proof}

\begin{example}\label{irrational-epsilon}
Let $r=10$ and let $L = 16H-5E$. Then $L$ is ample by \cite{Eckl}, see also \cite[Theorem
2.18]{H}. After normalizing, we have $\mu=3.2$.
Suppose that Conjecture \ref{MainConjecture} is true. Since $\mu  < 77/24\approx 3.208$, by Theorem
\ref{thm-Irrational}, there are no weakly submaximal curves for $L(\mu)$. Since
$\sqrt{L(\mu)^2} = \sqrt{0.24} \notin \QQ$, it follows that $\varepsilon(L(\mu))
\notin \QQ$. Hence $\varepsilon(L) \notin \QQ$.
\end{example}

\bibliographystyle{plain}

\begin{thebibliography}{15}
\bibitem{C} C. Ciliberto, Geometric aspects of polynomial interpolation in more variables and of Waring's problem, in {\it European Congress of Mathematics, Vol. I (Barcelona, 2000)}, 289--316, Progr. Math., 201, Birkh\"{a}user, Basel.

\bibitem{CM1} C. Ciliberto\ and\ R. Miranda, Degenerations of planar linear systems, J. Reine Angew. Math. {\bf 501} (1998), 191--220.

\bibitem{CM2} C. Ciliberto\ and\ R. Miranda, The Segre and Harbourne-Hirschowitz conjectures, in {\it Applications of algebraic geometry to coding theory, physics and computation (Eilat, 2001)}, 37--51, NATO Sci. Ser. II Math. Phys. Chem., 36, Kluwer Acad. Publ., Dordrecht.

\bibitem{Dem} J.-P. Demailly, Singular Hermitian metrics on positive line bundles, in {\it Complex algebraic varieties (Bayreuth, 1990)}, 87--104, Lecture Notes in Math., 1507, Springer, Berlin.

\bibitem{DKMS} M. Dumnicki, A. K\"{u}ronya, C. Maclean\ and\ T. Szemberg,
Rationality of Seshadri constants and the Segre-Harbourne-Gimigliano-Hirschowitz conjecture, Adv. Math. {\bf 303} (2016), 1162--1170.

\bibitem{Eckl} T. Eckl,
   Ciliberto-Miranda degenerations of $\CC\P^2$ blown up in $10$ points, J. Pure Appl. Alg. \textbf{215} (2011), 672--696.

\bibitem{EH} D. Eisenbud\ and\ J. Harris, {\it 3264 and all that---a second course in algebraic geometry}, Cambridge University Press, Cambridge, 2016.

\bibitem{G} A. Gimigliano, {\it On linear systems of plane
    curves}, Ph.D. Thesis, Queen's University, Kingston, (1987).

\bibitem{G1} A. Gimigliano, Our thin knowledge of fat points, in {\it The Curves Seminar at Queen's, Vol. VI (Kingston, ON, 1989)}, Exp. B, 50 pp, Queen's Papers in Pure and Appl. Math., 83, Queen's Univ., Kingston, ON.

\bibitem{H} K. Hanumanthu, Positivity of line bundles on general blow ups of $\Bbb{P}^2$, J. Algebra {\bf 461} (2016), 65--86.

\bibitem{HH} K. Hanumanthu\ and\ B. Harbourne, Single point Seshadri constants on rational surfaces, J. Algebra {\bf 499} (2018), 37--42.

\bibitem{H1} B. Harbourne, The geometry of rational surfaces and Hilbert functions of points in the plane, in {\it Proceedings of the 1984 Vancouver conference in algebraic geometry}, 95--111, CMS Conf. Proc., 6, Amer. Math. Soc., Providence, RI.

\bibitem{H2} B. Harbourne, Problems and progress: a survey on fat points in $\bold P^2$, in {\it Zero-dimensional schemes and applications (Naples, 2000)}, 85--132, Queen's Papers in Pure and Appl. Math., 123, Queen's Univ., Kingston, ON.

\bibitem{Har70} R. Hartshorne, {\it Ample subvarieties of algebraic varieties}, Lecture Notes in Mathematics, Vol. 156, Springer-Verlag, Berlin, 1970.

\bibitem{Hi}  A. Hirschowitz, Une conjecture pour la cohomologie des diviseurs sur les surfaces rationnelles g\'{e}n\'{e}riques, J. Reine Angew. Math. {\bf 397} (1989), 208--213.

\bibitem{Ogu} K. Oguiso, Seshadri constants in a family of surfaces, Math. Ann. {\bf 323} (2002), no.~4, 625--631.

\bibitem{Sa14} T. Sano, Seshadri constants on rational surfaces with
anticanonical pencils, J. Pure Appl. Alg. \textbf{218} (2014), 602--617.

\bibitem{S} B. Segre, Alcune questioni su insiemi finiti di punti in geometria algebrica, in {\it Atti Convegno Internaz. Geometria Algebrica (Torino, 1961)}, 15--33, Rattero, Turin.

\bibitem{SS}  B. Strycharz-Szemberg\ and\ T. Szemberg, Remarks on the Nagata conjecture, Serdica Math. J. {\bf 30} (2004), no.~2-3, 405--430.
\end{thebibliography}

\end{document}